\documentclass[12pt]{amsart}

\usepackage{amsfonts}
\usepackage{amsmath}
\usepackage{amsthm}

\usepackage{color}

\newcommand{\beq}{\begin{equation}}
\newcommand{\eeq}{\end{equation}}
\newcommand{\bea}{\begin{eqnarray}}
\newcommand{\eea}{\end{eqnarray}}
\newcommand{\beas}{\begin{eqnarray*}}
\newcommand{\M}{\mathbb M}
\newcommand{\Ho}{\mathcal H}
\newcommand{\V}{\mathcal V}
\newcommand{\W}{\mathcal W}
\newcommand{\eeas}{\end{eqnarray*}}

\newcommand{\Var}{\textrm{Var}}
\newcommand{\Ent}{\textrm{Ent}}

\newcommand{\ve}{\varepsilon}

\newtheorem{theorem}{Theorem}[section]

\newtheorem{assumption}[theorem]{Assumption}
\newtheorem{proposition}[theorem]{Proposition}
\newtheorem{corollary}[theorem]{Corollary}
\newtheorem{lemma}[theorem]{Lemma}
\newtheorem{remark}[theorem]{Remark}

\newcommand{\R}[1]{\mathbb{R}}     
\newcommand{\bM}{\mathbb{M}}

\newcommand{\ro}{\tilde{\rho}}

\title[Curvature-dimension inequalities]{Curvature-dimension estimates for the Laplace-Beltrami operator of a totally geodesic foliation}

\author{Fabrice Baudoin}
\address{Department of Mathematics\\Purdue University \\
West Lafayette, IN 47907 \\
USA
} \email[Fabrice Baudoin]{fbaudoin@purdue.edu}

\author{Michel Bonnefont}
\address{Institut de Math\'ematiques de Bordeaux \\
Universit\'e Bordeaux 1 \\
33405 Talence \\
FRANCE 
} \email[Michel Bonnefont]{michel.bonnefont@math.u-bordeaux1.fr}

\begin{document}

\maketitle

\begin{abstract}
We study Bakry-\'Emery type estimates for the Laplace-Beltrami operator of a totally geodesic foliation. In particular, we are interested in situations for which the $\Gamma_2$ operator may not be bounded from below but the horizontal Bakry-\'Emery curvature is. As we prove it, under a bracket generating condition, this weaker condition is enough to imply several  functional inequalities for the heat semigroup including the Wang-Harnack inequality and the log-Sobolev inequality. We also prove that, under proper additional assumptions, the generalized curvature dimension inequality introduced by Baudoin-Garofalo \cite{BG} is uniformly satisfied for a family of Riemannian metrics that  converge to the sub-Riemannian one.
\end{abstract}

\tableofcontents

\section{Introduction}

In the recent few years, there have been several works using Riemannian geometry tools to study sub-Laplacians. We refer to the survey \cite{survey} for an overview of those techniques. In the present work we somehow take the opposite stance and show how sub-Riemannian geometry can be used to study Laplacians.  More precisely, we shall be interested in the Laplace-Beltrami operator of a Riemannian foliation with totally geodesic leaves. Our main assumption will be that the horizontal distribution of the foliation is bracket-generating and the horizontal Bakry-\'Emery curvature of the Laplace-Beltrami operator  is bounded from below. However, we will not assume anything on the vertical Bakry-\'Emery curvature. As a consequence the $\Gamma_2$ operator does not need to be bounded from below. Surprisingly, even  under this weak condition, we are still able to obtain several important functional inequalities for the heat semigroup. We mention in particular the Wang-Harnack inequality and the associated criterion for the log-Sobolev inequality.

\

 In the second part of our work, we show that the Baudoin-Garofalo generalized curvature dimension inequality \cite{BG} for the horizontal Laplacian of the foliation can be seen as a uniform limit of curvature dimension estimates for the Laplace-Beltrami operator of the canonical variations of the metric of the foliation.
 
 \
 
 The paper is organized as follows. In Section 2 we introduce the geometric setting and establish Bochner's identities for the Laplace-Beltrami operator of a Riemannian foliation with totally geodesic leaves. These formulas are new and interesting in themselves. The main geometric novelty in those formulas is to separate the Bakry-\'Emery curvature of the Laplacian into two parts: an horizontal part and a vertical part.
 
 \
 
 In Section 3 we assume that the horizontal Bakry-\'Emery curvature is bounded from below and deduce several gradient bounds for the diffusion semigroup.  We are able to prove the Wang-Harnack inequality and deduce from it a criterion for the log-Sobolev inequality.
 
 \

In Section 4, we show that if additionally the vertical Bakry-\'Emery curvature is bounded from below, we get a uniform  family of curvature-dimension estimates for a one-parameter of Laplacians. This curvature dimension inequalities interpolate between the classical Bakry-\'Emery curvature dimension condition and the Baudoin-Garofalo curvature dimension inequality introduced in \cite{BG}.

\section{Bochner identities for the Laplace-Beltrami operator of a totally geodesic foliation}

The goal of the section will be to prove Bochner's type identities  for the Laplace-Beltrami operator of a totally geodesic Riemannian foliation. We refer to \cite {survey} for a detailed account about the geometry of such foliations.

\

We consider a smooth $n+m$ dimensional connected manifold $\M$ which is equipped with a Riemannian foliation with a complete bundle like  metric $g$ and totally geodesic leaves. We moreover assume that the metric $g$ is complete and that the horizontal distribution $\mathcal{H}$ of the foliation is Yang-Mills (see \cite{survey}).  We shall also assume that $\Ho$ is bracket-generating.

\

The $m$ dimensional sub-bundle $\mathcal{V}$ formed by vectors tangent to the leaves is referred  to as the set of \emph{vertical directions}. The sub-bundle $\mathcal{H}$ which is normal to $\mathcal{V}$ is referred to as the set of \emph{horizontal directions}.   The metric $g$ can be split as
\[
g=g_\mathcal{H} \oplus g_{\mathcal{V}},
\]
and  we introduce the one-parameter family of Riemannian metrics:
\[
g_{\varepsilon}=g_\mathcal{H} \oplus  \frac{1}{\varepsilon }g_{\mathcal{V}}, \quad \varepsilon >0.
\]
It is called the canonical variation of $g$. The Riemannian distance associated with $g_{\varepsilon}$ will be denoted by $d_\varepsilon$. Finally we denote by $\mu_\varepsilon$ the Riemannian  volume associated to $g_\varepsilon$.

\

The Laplace-Beltrami operator of the metric $g_\varepsilon$ is given by
\[
\Delta_\varepsilon =\Delta_{\mathcal{H}}+\varepsilon \Delta_{\mathcal{V}}, 
\]
where $\Delta_{\mathcal{H}}$ is the horizontal Laplacian of the foliation and $ \Delta_{\mathcal{V}}$ the vertical Laplacian.

The Bott connection which is defined in terms of the Levi-Civita connection $D$ of the metric $g$ by
\[
\nabla_X Y =
\begin{cases}
 ( D_X Y)_{\mathcal{H}} , \quad X,Y \in \Gamma^\infty(\mathcal{H}) \\
 [X,Y]_{\mathcal{H}}, \quad X \in \Gamma^\infty(\mathcal{V}), Y \in \Gamma^\infty(\mathcal{H}) \\
 [X,Y]_{\mathcal{V}}, \quad X \in \Gamma^\infty(\mathcal{H}), Y \in \Gamma^\infty(\mathcal{V}) \\
 ( D_X Y)_{\mathcal{V}}, \quad X,Y \in \Gamma^\infty(\mathcal{V})
\end{cases}
\]
where the subscript $\Ho$  (resp. $\mathcal{V}$) denotes the projection on $\mathcal{H}$ (resp. $\mathcal{V}$). Let us observe that for horizontal vector fields $X,Y$ the torsion $T(X,Y)$ is given by
\[
T(X,Y)=-[X,Y]_\V.
\]
Also observe that for $X,Y \in \Gamma^\infty(\mathcal{V})$ we actually have  $( D_X Y)_{\mathcal{V}}= D_X Y$ because the leaves are assumed to be totally geodesic. Finally,  it is easy to check that for every $\varepsilon >0$, the Bott connection satisfies $\nabla g_\varepsilon=0$. The horizontal Ricci curvature of $\nabla$ will be denoted $\mathbf{Ricci}_\mathcal{H}$ and the Ricci curvature of the leaves will be denoted $\mathbf{Ricci}_\mathcal{V}$.

For $Z \in \Gamma^\infty(T\M)$, there is a  unique skew-symmetric endomorphism  $J_Z:\mathcal{H}_x \to \mathcal{H}_x$ such that for all horizontal vector fields $X$ and $Y$,
\begin{align}\label{Jmap}
g_\mathcal{H} (J_Z (X),Y)= g_\mathcal{V} (Z,T(X,Y)).
\end{align}
where $T$ is the torsion tensor of $\nabla$. We then extend $J_{Z}$ to be 0 on  $\mathcal{V}_x$.  If $Z_1,\cdots,Z_m$ is a local vertical frame, the operator $\sum_{l=1}^m J_{Z_l}J_{Z_l}$ does not depend on the choice of the frame and is denoted by $\mathbf{J}^2$. For instance, if $\M$ is a K-contact manifold equipped with the Reeb foliation, then $\mathbf{J}$ is an almost complex structure, $\mathbf{J}^2=-\mathbf{Id}_{\mathcal{H}}$.

A simple computation (see for instance Theorem 9.70, Chapter 9 in \cite{Besse}) gives the following result for the Riemannian Ricci curvature of the metric $g_\varepsilon$.

\begin{lemma}\label{Ric}
Let us denote by $\mathbf{Ricci}_\varepsilon$ the Ricci curvature tensor of the Levi-Civita connection of the metric $g_\varepsilon$, then for every $X \in \Gamma^\infty(\mathcal{H})$ and $Z \in \Gamma^\infty(\mathcal{V})$,
\[
\mathbf{Ricci}_\varepsilon (Z,Z)=\mathbf{Ricci}_\mathcal{V} (Z,Z)+\frac{1}{4\varepsilon^2}\mathbf{Tr} ( J_Z^* J_Z)
\]
\[
\mathbf{Ricci}_\varepsilon (X,Z)=0
\]
\[
\mathbf{Ricci}_\varepsilon (X,X)=\mathbf{Ricci}_\mathcal{H} (X,X)-\frac{1}{2\varepsilon} \| \mathbf{J} X \|^2.
\]
\end{lemma}

We explicitly note that $\mathbf{Ricci}_\varepsilon (X,Z)=0$ is due to the fact that the foliation is assumed to be of Yang-Mills type.

We denote 
\[
\Gamma_2^\mathcal H (f) = \frac{1}{2} \Delta_{\mathcal{H}} \| \nabla_\mathcal H f\|^2 -\langle \nabla_{\mathcal{H}} \Delta_{\mathcal{H}} f, \nabla_{\mathcal{H}} f\rangle.
\]
Our first results are Bochner's type identities for the operator $\Delta_\varepsilon$.

\begin{proposition}\label{boch}
For every $f \in C^\infty(\M)$,
\[
\frac{1}{2} \Delta_{\varepsilon} \| \nabla_\mathcal H f\|^2 -\langle \nabla_{\mathcal{H}} \Delta_{\varepsilon} f, \nabla_{\mathcal{H}} f\rangle=\Gamma_2^\mathcal H (f) +\varepsilon \| \nabla_{\V,\Ho}^2 f \|^2
\]
and
\[
\frac{1}{2} \Delta_{\varepsilon} \| \nabla_\mathcal V f\|^2 -\langle \nabla_{\mathcal{V}} \Delta_{\varepsilon} f, \nabla_{\mathcal{V}} f\rangle=\varepsilon  \| \nabla_{\V}^2 f \|^2  + \varepsilon \mathbf{Ric}_\V (\nabla_\V f , \nabla_\V f) + \| \nabla_{\Ho,\V}^2 f \|^2
\]
\end{proposition}

\begin{proof} Since the statement is local, we can assume that the Riemannian foliation comes from  a Riemannian submersion with totally geodesic fibers. We fix  $x \in \M$ throughout the proof and prove the identities at the point $x$.

Let $X_1,\cdots, X_n$ be a local orthonormal horizontal frame around $x$ consisting of basic vector fields for the submersion. We can assume that, at $x$, $\nabla_{X_i} X_j =0$. Let now $Z_1,\cdots,Z_m$ be a local orthonormal vertical frame around $x$, such that at $x$, $\nabla_{Z_l}Z_m=0$. Since $X_i$ is basic, the vector field $[X_i,Z_m]$ is tangent to the leaves. We  write the structure constants in that local frame:
\[
[X_i,X_j]=\sum_{k=1}^n \omega_{ij}^k X_k +\sum_{k=1}^m \gamma_{ij}^k Z_k
\]
\[
[X_i,Z_k]=\sum_{j=1}^m \beta_{ik}^j Z_j,
\]
and observe that at the center $x$ of the frame, we have $\omega_{ij}^k =0$. Moreover, since  the submersion has totally geodesic fibers we have the skew-symmetry,
\[
\beta_{ik}^j =-\beta_{ij}^k.
\]
We can also assume that, at the center $x$, $\beta_{ij}^k=0$. 

Observe that we have at $x$
\[
\Delta_{\mathcal H} =\sum_{i=1}^n X_i^2 
\]
and
\[
\Delta_{\mathcal V}= \sum_{j=1}^m Z_j^2.
\]
We obviously have
\[
\frac{1}{2} \Delta_{\varepsilon} \| \nabla_\mathcal H f\|^2 -\langle \nabla_{\mathcal{H}} \Delta_{\varepsilon} f, \nabla_{\mathcal{H}} f\rangle=\Gamma_2^\mathcal H (f) +\varepsilon  \left( \frac{1}{2} \Delta_{\mathcal V} \| \nabla_\mathcal H f\|^2 -\langle \nabla_{\mathcal{H}} \Delta_{\mathcal V} f, \nabla_{\mathcal{H}} f\rangle \right).
\]
So, we have to prove that
\[
 \frac{1}{2} \Delta_{\mathcal V} \| \nabla_\mathcal H f\|^2 -\langle \nabla_{\mathcal{H}} \Delta_{\mathcal V} f, \nabla_{\mathcal{H}} f\rangle=\| \nabla_{\V,\Ho}^2 f \|^2.
\]
At the center $x$ of the frame, we easily see that
\begin{align*}
 \frac{1}{2} \Delta_{\mathcal V} \| \nabla_\mathcal H f\|^2 -\langle \nabla_{\mathcal{H}} \Delta_{\mathcal V} f, \nabla_{\mathcal{H}} f\rangle & =\sum_{i=1}^n \sum_{j=1}^m (Z_m X_i f)^2 +\sum_{i=1}^n (X_i f) [ \Delta_{\mathcal V} ,X_i]f \\
  &=\sum_{i=1}^n \sum_{j=1}^m (Z_m X_i f)^2 \\
  &=\| \nabla_{\V,\Ho}^2 f \|^2.
\end{align*}
The second identity is proved in a siimilar way. We have
\[
\frac{1}{2} \Delta_{\varepsilon} \| \nabla_\mathcal V f\|^2 -\langle \nabla_{\mathcal{V}} \Delta_{\varepsilon} f, \nabla_{\mathcal{V}} f\rangle=\frac{1}{2} \Delta_{\mathcal H} \| \nabla_\mathcal V f\|^2 -\langle \nabla_{\mathcal{V}} \Delta_{\mathcal H} f, \nabla_{\mathcal{V}} f\rangle +\varepsilon \left(\frac{1}{2} \Delta_{\mathcal V} \| \nabla_\mathcal V f\|^2 -\langle \nabla_{\mathcal{V}} \Delta_{\mathcal V} f, \nabla_{\mathcal{V}} f\rangle \right).
\]
From the Bochner's identity on the leaves, we have
\[
 \frac{1}{2} \Delta_{\mathcal V} \| \nabla_\mathcal V f\|^2 -\langle \nabla_{\mathcal{V}} \Delta_{\mathcal V} f, \nabla_{\mathcal{V}} f\rangle= \| \nabla_{\V}^2 f \|^2  +  \mathbf{Ric}_\V (\nabla_\V f , \nabla_\V f) 
\]
and it is easy to see that in the local frame, at $x$ we have
\begin{align*}
\frac{1}{2} \Delta_{\mathcal H} \| \nabla_\mathcal V f\|^2 -\langle \nabla_{\mathcal{V}} \Delta_{\mathcal H} f, \nabla_{\mathcal{V}} f\rangle & = \sum_{i=1}^n \sum_{j=1}^m (X_i Z_m f)^2 \\
 &= \| \nabla_{\Ho,\V}^2 f \|^2.
\end{align*}
It is worth pointing out that
\[
 \| \nabla_{\Ho,\V}^2 f \|^2=\| \nabla_{\V,\Ho}^2 f \|^2,
\]
since at the center of the frame
\[
\sum_{i=1}^n \sum_{j=1}^m (Z_m X_i f)^2=\sum_{i=1}^n \sum_{j=1}^m (X_i Z_m f)^2.
\]
\end{proof}

%
%

\section{Functional inequalities with the horizontal gradient}

In this section, in addition to the assumptions of previous section, we also assume that for every  $X \in \Gamma^\infty(\mathcal{H})$,
\[
\mathbf{Ricci}_\mathcal{H} (X,X) \ge \rho_1 \| X \|^2, \quad  \| \mathbf{J} X \|^2 \le \kappa \| X \|^2.
\]
In particular, according to Lemma \ref{Ric}, we have for every $\varepsilon >0$ and horizontal vector field $X$
\[
\mathbf{Ricci}_\varepsilon (X,X) \ge \left( \rho_1 - \frac{\kappa}{ 2\varepsilon} \right),
\]
with $\kappa \ge 0$ and $\rho_1 \ge 0$. However, no assumption is made on $\mathbf{Ricci}_\varepsilon (Z,Z)$ when $Z$ is vertical

We fix $\varepsilon >0$ in the sequel and consider
\[
\mathcal{T}_2(f)=\frac{1}{2} \Delta_{\varepsilon}  \| \nabla_\mathcal{H} f \|^2 -\langle \nabla_\Ho  \Delta_\varepsilon f, \nabla_\Ho f \rangle.
\]

\begin{proposition}
For every $f \in C^\infty(\M)$,
\begin{align}\label{CD8}
\mathcal{T}_2(f) \ge \left( \rho_1 -\frac{\kappa}{\varepsilon} \right)  \| \nabla_\mathcal{H} f \|^2
\end{align}
\end{proposition}

\begin{proof}
From Proposition \ref{boch}, we have
\[
\mathcal{T}_2(f) =\Gamma_2^\mathcal H (f) +\varepsilon \| \nabla_{\V,\Ho}^2 f \|^2.
\]
As a consequence of \cite{BKW} we have
\[
\Gamma_2^\mathcal H (f) +\varepsilon \| \nabla_{\V,\Ho}^2 f \|^2  \ge \left( \rho_1 -\frac{\kappa}{\varepsilon} \right)  \| \nabla_\mathcal{H} f \|^2.
\]
This concludes the proof.
\end{proof}

\subsection{Gradient bounds for the heat semigroup}

We now investigate the consequences of \eqref{CD8} in terms of functional inequalities for the diffusion semigroup generated by $\Delta_\varepsilon$. We denote by $P_t^\varepsilon$ the semigroup generated by $\Delta_\varepsilon$ and denote by $ \Gamma_{\varepsilon}$ the \textit{carr\'e du champ} of $\Delta_\varepsilon$. Observe that our assumptions do not imply any lower bound on the $\Gamma_2$ operator of $\Delta_\varepsilon$.

 The following assumptions will be in force throughout this section:

\begin{assumption}

\

\begin{enumerate}
\item For every $t \ge 0$, $P_t^\varepsilon 1 =1$;
\item For every $t \ge 0$ and $f \in C^\infty_0(\M)$, 
\[
\sup_{0 \le s \le t} \|  \Gamma_{\varepsilon} (P_s^\ve f ) \|_\infty < \infty.
\]
\end{enumerate}
\end{assumption}

\

The first gradient bound we have is the following.

 \begin{proposition}
 For every $f \in C_0^\infty (\M)$, $f \ge 0$, we have for $t \ge 0$,
\begin{equation}\label{e:}
 (P_t^{\varepsilon} f )  \,  (\| \nabla_\Ho  \ln P_{t}^{\varepsilon} f \|^2 ) \leq  e^{- 2\left( \rho_1 -\frac{\kappa}{\varepsilon} \right) t}   P_{t}^{\varepsilon} (f \| \nabla_\Ho \ln f \|^2 ) .
\end{equation}
\end{proposition}
\begin{proof}
The argument is close to the one we use in \cite{BB}, so we only sketch the proof.
We fix $t>0$ and denote
\[
 \phi(s):=  e^{-2 \left( \rho_1 -\frac{\kappa}{\varepsilon} \right)s}( P_{t-s}^{\varepsilon} f )  \,  \| \nabla_\Ho  \ln P_{t-s}^{\varepsilon} f \|^2.
 \]
Since for every smooth $g$,  $$\langle \nabla_\mathcal{H} g , \nabla_\mathcal{H} \Gamma_\varepsilon (g) \rangle= \Gamma_\varepsilon ( g, \| \nabla_\mathcal{H} g \|^2  ),$$
one has
\begin{align*}
  \frac{d}{ds}\phi(s)+ \Delta_\ve \phi (s) = 2 ( P_{t-s}^{\varepsilon} f ) \mathcal{T}_2( P_{t-s}^{\varepsilon} f ) - 2 \left( \rho_1 -\frac{\kappa}{\varepsilon} \right) \phi (s) \ge 0.
\end{align*}
It is now easy to conlude from a parabolic comparison theorem.
\end{proof}

Of course, this also implies that

\begin{proposition}\label{p:reg}
 For every $f \in C_0^\infty (\M)$,  we have for $t \ge 0$
\begin{equation}\label{e:2}
  \| \nabla_\Ho  P_{t}^{\varepsilon} f \|^2  \leq  e^{- 2 \left( \rho_1 -\frac{\kappa}{\varepsilon} \right) t}   P_{t}^{\varepsilon} (\| \nabla_\Ho f \|^2 ) .
\end{equation}
and
\begin{equation}\label{e:reg_0}
  \| \nabla_\Ho  P_{t}^{\varepsilon} f \|_\infty \leq  e^{-  \left( \rho_1 -\frac{\kappa}{\varepsilon} \right) t}  \| \nabla_\Ho f \|_\infty.
\end{equation}
\end{proposition}
We also have the following reverse log-Sobolev inequality: 
\begin{proposition}\label{reverse_logsob}
 For every $f \in C_0^\infty (\M)$, $f \ge 0$, we have for $t \ge 0$
\begin{equation}\label{e:RevLS}
 P_t^{\varepsilon} f  \| \nabla_\Ho   \ln  P_{t}^{\varepsilon} f  \|^2 
 \le \frac{  2 \left( \rho_1 -\frac{\kappa}{\varepsilon} \right)} {e^{ 2 \left( \rho_1 -\frac{\kappa}{\varepsilon} \right) t}-1} 
\, \left( P_t^{\varepsilon} (f\ln f) -  P_t^{\varepsilon}(f) \ln  P_t^{\varepsilon}(f) \right).
\end{equation}
\end{proposition}
\begin{proof}
One has:
\begin{align*}
 P_t^{\varepsilon} (f\ln f) -  P_t^{\varepsilon}(f) \ln  P_t^{\varepsilon}(f)
   &= \int_0^t   \frac{d}{ds} P_s^{\varepsilon} \left( ( P_{t-s}^{\varepsilon} f ) \ln ( P_{t-s}^{\varepsilon} f )\right) ds\\
                                    &=  \int_0^t  P_s^{\varepsilon} \left(  ( P_{t-s}^{\varepsilon} f ) \Gamma_{\varepsilon}( \ln  P_{t-s}^{\varepsilon} f ) \right) ds\\
                                    &\ge \int_0^t  P_s^{\varepsilon} \left(  ( P_{t-s}^{\varepsilon} f ) \| \nabla_\Ho   \ln  P_{t-s}^{\varepsilon} f  \|^2 \right) ds\\
                                    &\ge \int_0^t e^{  2\left( \rho_1 -\frac{\kappa}{\varepsilon} \right) s}   ( P_t^{\varepsilon}  f)  \| \nabla_\Ho   \ln  P_{t}^{\varepsilon} f  \|^2  ds\\
                                    &=  \left( \frac{ e^{2 \left( \rho_1 -\frac{\kappa}{\varepsilon} \right) t} -1 } { 2\left( \rho_1 -\frac{\kappa}{\varepsilon} \right) } \right)
                                    (P_t^{\varepsilon}  f)  \| \nabla_\Ho   \ln  P_{t}^{\varepsilon} f  \|^2.
\end{align*}

\end{proof}

Similarly we get the following reverse Poincar\'e inequality:  

\begin{proposition}
For every $f \in C_0^\infty (\M)$,  we have for $t \ge 0$,
\begin{equation}\label{e:RevP}
   \| \nabla_\Ho   P_{t}^{\varepsilon} f  \|^2 
 \le \frac{  \left(  \rho_1 -\frac{\kappa}{\varepsilon} \right)} {e^{ 2 \left( \rho_1 -\frac{\kappa}{\varepsilon} \right) t}-1 } 
\, \left( P_t^{\varepsilon} (f^2) -  P_t^{\varepsilon}(f)^2 \right).
\end{equation}

and
\begin{equation}\label{e:reg}
  \| \nabla_\Ho  P_{t}^{\varepsilon} f \|_\infty \leq \sqrt{ \frac{  \left(  \rho_1 -\frac{\kappa}{\varepsilon} \right)} { e^{ 2 \left( \rho_1 -\frac{\kappa}{\varepsilon} \right) t} -1 } }  \,  \|  f \|_\infty.
\end{equation}
\end{proposition}

We then prove the following Wang-Harnack inequality:

\begin{proposition}\label{Wang inequality}
Let $\alpha>1$. For $f \in L^\infty(\bM)$,   $f \ge 0$, $t>0$, $x,y \in \bM$,
\[
(P_t^\ve f)^\alpha (x) \leq P_t^\ve(f^\alpha) (y) \exp \left(  \frac{\alpha}{4(\alpha-1)}
\left( \frac  { 2\left( \rho_1 -\frac{\kappa}{\varepsilon} \right) } {e^{ 2\left( \rho_1 -\frac{\kappa}{\varepsilon} \right) t} -1} \right)  d_\Ho^2(x,y)\right).
\]
where $d_\Ho$ is the sub-elliptic distance $d_\Ho$ associated to $\nabla_\Ho$, that is
\[
d_\Ho (x,y) =\sup \{ | f(x)-f(y) |, f \in C_0^\infty(\M), \|  \nabla_\Ho f \|_\infty \le 1 \}.
\]
\end{proposition}

\begin{proof}In this proof to simplify the notation, we write $P_t$ for $P_t^\ve$ and set  $\ro:=\left( \rho_1-\frac{\kappa}{\varepsilon}\right)$.

 Consider a  curve $\gamma:[0,T] \rightarrow \bM$ such that $\gamma(0)=x$, $\gamma(T)=y$ and which is subunit with respect to the sub-elliptic distance associated to $\nabla_\Ho$.  
 Let $\alpha >1$ and $\beta(s)=1+(\alpha-1)\frac{s}{T}$, $0 \le s \le T$.  Let
\[
\psi (s)=\frac{\alpha}{\beta (s)}  \ln P_t f^{\beta (s)} (\gamma(s)), \quad 0 \le s \le T.
\]
where $t>0$ is fixed. Differentiating with respect to $s$ and using  then Proposition  \ref{reverse_logsob} yields
\begin{align*}
\psi' (s) & \ge \frac{\alpha (\alpha-1)}{T \beta(s)^2} \frac{P_t (f^{\beta(s)} \ln f^{\beta(s)}) -(P_t f^{\beta(s)}) \ln P_t f^{\beta(s)}  }{P_t f^{\beta(s)} }
-\frac{\alpha}{\beta (s)}  \|  \nabla_\Ho ( \ln P_t f^{\beta(s)}) \|  \\
 & \ge \frac{\alpha (\alpha-1)  }{T\beta(s)^2}   \left( \frac   { e^{ 2 \ro t}-1 } {2 \ro} \right) 
 \|  \nabla_\Ho ( \ln P_t f^{\beta(s)}) \|^2 -\frac{\alpha}{\beta (s)}   \|  \nabla_\Ho ( \ln P_t f^{\beta(s)}) \|.
\end{align*}
Now, for every $\lambda >0$,
\[
  -  \|  \nabla_\Ho ( \ln P_t f^{\beta(s)}) \| \ge -\frac{1}{2\lambda^2}  \|  \nabla_\Ho ( \ln P_t f^{\beta(s)}) \|^2-\frac{\lambda^2}{2} .
\]
If we choose
\[
\lambda^2=\frac{ T \beta (s) }{2(\alpha -1) }  \left( \frac  { 2\ro } { e^{ 2\ro t} -1} \right), 
\]
we infer that
\[
\psi' (s) \ge - \frac{\alpha  T  }{4(\alpha -1) }  \left( \frac  { 2\ro } { e^{ 2\ro t} -1} \right).
\]

Integrating from $0$ to $L$ yields
\[
\ln P_t (f^\alpha) (y) -\ln (P_t f)^\alpha (x) \ge - \frac{\alpha   }{4(\alpha -1) }\left( \frac  { 2\ro } { e^{ 2\ro t} -1} \right) \, T^2 .
\]
Minimizing then  $T^2$ over the set of subunit curves such that $\gamma(0)=x$ and $\gamma(T)=y$ gives the claimed result.

\end{proof}

An easy consequence of the Wang inequality of Proposition \ref{Wang inequality} is the following log-Harnack inequality.  
\begin{proposition}\label{log_harnack}
For $f \in L^\infty(\bM)$,   $\inf f  > 0$, $t>0$, $x,y \in \bM$,
$$
P^\varepsilon_t(\ln f)(x) \leq \ln P^\varepsilon_t(f)(y) + \frac{1}{4} \left( \frac  { 2\left( \rho_1 -\frac{\kappa}{\varepsilon} \right) } {e^{ 2 \left( \rho_1 -\frac{\kappa}{\varepsilon} \right) t} -1} \right) d_\Ho^2(x,y).
$$
\end{proposition}

The proof of this result appears in Section 2 of \cite{W3} where  a general study of these Harnack inequalities is done. 

When $\mu_\ve$ is a probability measure, the above log-Harnack inequalities implies the following lower bound for the heat kernel $p_t^\varepsilon$ of $P_t^\varepsilon$.
\begin{corollary}
Assume that $\mu_\ve$ is a probability measure, then for $t>0$, $x,y \in \bM$,
$$
p_{t}^\ve(x,y) \geq \exp\left(-  \frac{1}{4} \left( \frac  {2 \left( \rho_1 -\frac{\kappa}{\varepsilon} \right) } {e^{ \left( \rho_1 -\frac{\kappa}{\varepsilon} \right) t} -1} \right) d_\Ho^2(x,y))\right).
$$
\end{corollary}

\subsection{Log-Sobolev inequality}

With Wang-Harnack inequality in hands, we can prove a log-Sobolev inequality provided the exponential integrability of the square distance function. We refer to \cite{SOB} and \cite{W2} for the analogue of this result when the classical Bakry-\'Emery criterion is satisfied.

\begin{theorem}\label{Wang}
Assume that the measure $\mu_\varepsilon$ is a probability measure and that there exists $\lambda >  \frac{\left(\rho_1-\frac{\kappa}{\varepsilon}\right)_-}{2}$ such that:  
\[
\int_\bM e^{\lambda d_\Ho^2 (x_0,x )} d\mu_\ve(x) <+\infty,
\]
for some some $x_0 \in \bM$.  Then, there is a constant $C>0$ such that for every function $f  \in C^ \infty_0(\bM)$,
\[
\int_\bM f^2 \ln f^2 d\mu_\varepsilon -\int_\bM f^2 d\mu_\varepsilon \ln  \int_\bM f^2 d\mu_\varepsilon \le C \int_\bM \Gamma_\varepsilon(f) d\mu_\varepsilon.
\]
\end{theorem}

\begin{proof}
For simplicity, as before we write $P_t$ for $P_t^\ve$ and set  $\ro:=\left( \rho_1-\frac{\kappa}{\varepsilon}\right)$.
Let $\alpha >1$ and $f \in L^\infty(\bM)$, $ f \ge 0$. From Proposition \ref{Wang inequality}, by integrating with respect to $y$, we have
\begin{align*}
\int_\bM f^\alpha (y) d\mu_\ve  (y) & \ge (P_t f)^\alpha (x)   \int_\bM \exp\left( -\frac{\alpha}{4(\alpha-1)}   \left( \frac  {  2\ro } {e^{  2\ro t} -1} \right)  d_\Ho^2(x,y) \right) d\mu_\ve(y) \\
 &  \ge (P_t f)^\alpha (x)   \int_{B(x_0,1)} \exp\left( -\frac{\alpha}{4(\alpha-1)}   \left( \frac  { 2 \ro } {e^{  2\ro t} -1} \right)  d_\Ho^2(x,y) \right) d\mu_\ve(y) \\
 & \ge \mu_\ve( B(x_0,1)) (P_t f)^\alpha (x) \exp\left( -\frac{\alpha}{4(\alpha-1)}   \left( \frac  {  2\ro } {e^{  2\ro t} -1} \right) (d_\Ho^2(x_0,x)+1) \right).
\end{align*}
As a consequence, we get
\[
(P_t f) (x) \le \frac{1}{ \mu_\ve( B_\Ho(x_0,1))^\frac{1}{\alpha}}   \exp\left( \frac{1}{2(\alpha-1)}   \left( \frac  {  \ro } {e^{ 2 \ro t} -1} \right)  (d_\Ho^2(x_0,x)+1) \right) \| f \|_{L^\alpha} .
\]
Therefore if
\[
\int_\bM e^{\lambda d_\Ho^2 (x_0,x )} d\mu_\ve(x) <+\infty,
\]
for some $x_0 \in \bM$ and $\lambda > \frac{\ro_-}{2}$, then we can find $1< \alpha <\beta$  and $t>0$ such that
\[
\| P_t f \|_{L^\beta} \le C_{\alpha,\beta} \| f \|_{L^\alpha} ,
\]
for some constant $C_{\alpha,\beta}$.  This implies the supercontractivity of the semigroup $(P_t)_{t \ge 0}$. We deduce therefore from Gross' theorem (see \cite{bakry-stflour}) that  a defective logarithmic Sobolev inequality 
(with the full gradient)
holds; that is: there exist two constants $A,B>0$  such that
\[
\int_\bM f^2 \ln f^2 d\mu_\varepsilon -\int_\bM f^2 d\mu_\varepsilon \ln  \int_\bM f^2 d\mu^\varepsilon \le A \int_\bM \Gamma_\varepsilon(f) d\mu+B\int_{\bM} f^2 d\mu_\varepsilon, \quad f \in  C^ \infty_0(\bM).
\]
Now, since moreover the heat kernel is positive and the invariant measure a probability, the uniform positivity improving property (see \cite{A}, Theorem 2.11) implies that $L$ admits a spectral gap.
or equivalently that a Poincar\'e inequality is satisfied. It is then classical (see \cite{SOB}), that the conjunction of a spectral gap and a defective logarithmic Sobolev inequality implies the log-Sobolev inequality (i.e. we may actually take $B=0$ in the above inequality).  
\end{proof}

\subsection{Poincar\'e inequality}
 In the case where the curvature parameter  $\left(\rho_1 -\frac{\kappa}{\varepsilon}\right)$ is positive, we get an explicit constant for the spectral gap.

\begin{proposition}
Assume that $\mu_\varepsilon$ is a finite measure and that the sub-elliptic distance $d_\Ho$ is square integrable; that is 
 \[
  \int_\M d_\Ho(x_0,y)^2 d\mu_\ve(y) <+ \infty 
 \]
for some $x_0\in \M$.
 Assume  moreover that $\rho_1 -\frac{\kappa}{\varepsilon}>0$,  
 Then
the following Poincar\'e inequality holds:
 \[
  \Var_{\mu_\varepsilon} (f) \leq \frac{1}{ \left(\rho_1 -\frac{\kappa}{\varepsilon}\right) } \int_\bM \Gamma_\varepsilon (f,f) d\mu_\varepsilon.
 \]

\end{proposition}

\begin{proof} Take $f$ a smooth function with compact support. 
Integrating Proposition \ref{p:reg} along a geodesic between $x,y\in\bM$ , we infer that:
\[
 |P_t^\ve f(x) -P_t^\ve f (y)| \leq e^{-  \left(\rho_1 -\frac{\kappa}{\varepsilon}\right)t} d_\Ho(x,y) \| \nabla_\Ho f \|_\infty.
\]

 \begin{align*}
  \Var_{\mu_\varepsilon} (P_t^\ve f) &= \frac{1}{2}  \int_\bM \int_\bM ( P_t^\ve f(x)- P_t^\ve f(y) )^2 d\mu_\varepsilon(x) d\mu_\varepsilon(y)\\
                                     & \leq  \frac{1}{2} e^{- 2 \left(\rho_1 -\frac{\kappa}{\varepsilon}\right)t} \| \nabla_\Ho f \|_\infty^2 \int_\bM \int_\bM d_\Ho(x,y)^2 d\mu_\varepsilon(x) d\mu_\varepsilon(y)\\
                                     &=C_f \, e^{- 2 \left(\rho_1 -\frac{\kappa}{\varepsilon}\right)t}
 \end{align*} for some constant $C_f$ depending on $f$.
Now it is well known that this is enough to imply the desired Poincar\'e inequality (see for example \cite{CGZ}, Lemma 2.12). 
\end{proof}

\subsection{A bound of the entropy of the semigroup by  the $L^2$-Wasserstein distance}

The following result is an analogue of the Lemma 4.2 in \cite{BGL} (see also \cite{OV2}). 
\begin{proposition}\label{entropy_wasserstein}
Let $f$ and $g$ be  non negative functions on $\bM$ such that $\int_\bM f d\mu=\int_\bM g d\mu= 1$ and set $d\nu_0= g d\mu$ et $d\nu_1=gd\mu$. Then, for any $t>0$,
$$
\Ent_{\mu_\ve}( P^\ve_t f) \leq \Ent_{\mu_\ve}( g) +\frac{1}{2} \frac  { \left( \rho_1 -\frac{\kappa}{\varepsilon} \right) } {e^{ 2\left( \rho_1 -\frac{\kappa}{\varepsilon} \right) t} -1} \W_\Ho(\nu_0,\nu_1)^2.
$$
where $\W_\Ho$ is the 2-Wasserstein distance with respect to the sub-elliptic distance $d_\Ho$.
\end{proposition}
\begin{proof}As before we write $P_t$ for $P_t^\ve$ and set  $\ro:=\left( \rho_1-\frac{\kappa}{\varepsilon}\right)$.
Let $t>0$ and $f$ and $g$ be  positive functions on $\bM$ such that $\int_\bM f d\mu=\int_\bM g d\mu= 1$. The log-Harnack inequality of Proposition \ref{log_harnack} applied to the function $P_tf$ gives then:

$$
P_t(\ln P_tf)(x)\leq \ln P_{2t} (f)(y) + \frac{1}{s} d_\Ho^2(x,y) .
$$
with 
$$ s=4 \left( \frac {e^{2\ro t} -1} { 2\ro }  \right).
$$
For $x$ fixed, by taking the infimum with respect to $y$ on the right hand side of the last inequality, we obtain
$$
P_t(\ln P_t f) (x) \leq  Q_s (\ln P_{2t}f) (x)
$$ 

where $Q_s$ is the infimum-convolution semigroup:
$$Q_s (\phi)(x)= \inf_{y\in \bM} \left\{ \phi(y)+ \frac{1}{s} d_ \Ho(x,y)^2 \right\}.$$

By Jensen inequality, one has
$$
\int_\bM \ln P_{2t}f \, g \, d{\mu_\ve} - \int_\bM g\ln g d\mu_\ve=\int_\bM g \ln \left(\frac{P_{2t}f}{g}\right)  d\mu_\ve \leq \ln \left(  \int_\bM P_{2t} f d\mu_\ve\right) =0;
$$
thus, using symmetry, one finally gets:
\begin{align*}
\Ent_{\mu_\ve} (P_t f) &= \int_\bM f P_{t} (\ln P_t f) d\mu_\ve\\
   & \leq  \int_\bM   f Q_s (\ln P_{2t}f) d\mu_\ve\\
   &  \leq \int_\bM   Q_s (\ln P_{2t}f) d\nu_1  - \int_\bM  g \, \ln P_{2t}f d\mu_\ve + \Ent_{\mu_\ve}(g)\\
   & \leq  \Ent_{\mu_\ve}(g) + \sup_{\psi} \left\{ \int_\bM Q_s (\psi)  d\nu_1 -\int_\bM \psi d\nu_0 \right\}
\end{align*}
where the supremum is taken over all bounded mesurable functions  $\psi$.  By Monge-Kantorovich duality, 
$$
\sup_{\psi} \left\{ \int_\bM Q_s (\psi) d\nu_1 -\int_\bM \psi d\nu_0 \right\} = \inf_\Pi \int_\bM T(x,y) d\Pi (x,y)
$$
where the infimum is taken over all coupling of $\nu_0$ and $\nu_1$ 
 and where the cost $T$ is just 
$$
T(x,y)= \frac{1}{s} d_\Ho^2(x,y).
$$
Therefore the latter infimum is equal to $\frac{1}{s} \W_\Ho(\nu_0,\nu_1)^2$.
\end{proof}

\section{Generalized curvature dimension estimates for the Laplace-Beltrami operator}

In this section we  assume  that for every  $X \in \Gamma^\infty(\mathcal{H})$,
\[
\mathbf{Ricci}_\mathcal{H} (X,X) \ge \rho_1 \| X \|^2, \quad  \| \mathbf{J} X \|^2 \le \kappa \| X \|^2,
\]
and for every $Z \in \Gamma^\infty(\mathcal{V})$
$$ -\frac{1}{4} \mathbf{Tr}_\mathcal{H} (J^2_{Z})\ge \rho_2 \| Z \|^2_\mathcal{V}, \quad   \mathfrak{Ric}_{\mathcal{V}} (Z,Z)  \ge  \rho_3 \| Z \|^2.$$

We denote
\[
\Gamma_\varepsilon (f)= \| \nabla_\Ho f\|^2 +\varepsilon \| \nabla_\V f \|^2
\]
\[
\Gamma_{2,\varepsilon}  (f) = \frac{1}{2} \Delta_{\varepsilon} \Gamma_\varepsilon (f) -\Gamma_\varepsilon ( f, \Delta_\varepsilon f),
 \]
 \[
 \Gamma^\V_{2,\varepsilon}  (f) = \frac{1}{2} \Delta_{\varepsilon} \| \nabla_\V f \|^2  - \langle \nabla_\V f, \nabla_\V \Delta_\varepsilon f\rangle.
 \]
 
The following proposition shows that under the previous assumptions, the generalized curvature dimension inequality introduced in \cite{BG} is uniformly satisfied, at least when $\rho_1,\rho_3 \ge 0$,  for the family of operators $\Delta_\varepsilon$, $\varepsilon \ge 0$. 

\begin{proposition}
For every $f \in C^\infty(\M)$, and every $\nu \ge -\varepsilon$,
\begin{align*}
 & \Gamma_{2,\varepsilon}  (f) +\nu \Gamma^\V_{2,\varepsilon} (f) \\
  \ge & \frac{1}{n+\frac{m\varepsilon}{\nu+\varepsilon}} ( \Delta_\varepsilon f )^2+\left( \rho_1-\frac{\kappa}{2\varepsilon+\nu} \right) \Gamma_\varepsilon (f) +\left( \rho_2 -\rho_1 \varepsilon +\frac{\kappa \varepsilon}{2\varepsilon +\nu} +\rho_3 \varepsilon (\nu+\varepsilon) \right)\Gamma^\V (f).
\end{align*}
\end{proposition}

\begin{proof}
From Proposition \ref{boch}, we have
\[
\frac{1}{2} \Delta_{\varepsilon} \| \nabla_\mathcal H f\|^2 -\langle \nabla_{\mathcal{H}} \Delta_{\varepsilon} f, \nabla_{\mathcal{H}} f\rangle=\Gamma_2^\mathcal H (f) +\varepsilon \| \nabla_{\V,\Ho}^2 f \|^2
\]
and
\[
\frac{1}{2} \Delta_{\varepsilon} \| \nabla_\mathcal V f\|^2 -\langle \nabla_{\mathcal{V}} \Delta_{\varepsilon} f, \nabla_{\mathcal{V}} f\rangle=\varepsilon  \| \nabla_{\V}^2 f \|^2  + \varepsilon \mathbf{Ric}_\V (\nabla_\V f , \nabla_\V f) + \| \nabla_{\Ho,\V}^2 f \|^2.
\]
Therefore we have
\begin{align*}
 & \Gamma_{2,\varepsilon}  (f) +\nu \Gamma^\V_{2,\varepsilon} (f) \\
 =& \frac{1}{2} \Delta_{\varepsilon}( \| \nabla_\Ho f \|^2+(\varepsilon+\nu) \| \nabla_\V f \|^2 ) - \langle (\nabla_\Ho +(\varepsilon+\nu) \nabla_\V) f, (\nabla_\Ho +(\varepsilon+\nu) \nabla_\V) \Delta_\varepsilon f\rangle \\
 =&\Gamma_2^\mathcal H (f) +(2\varepsilon +\nu) \| \nabla_{\V,\Ho}^2 f \|^2+\varepsilon (\varepsilon +\nu)  \| \nabla_{\V}^2 f \|^2+\varepsilon (\varepsilon +\nu) \mathbf{Ric}_\V (\nabla_\V f , \nabla_\V f).
\end{align*}
From \cite{BKW}, we have
\[
\Gamma_2^\mathcal H (f) +(2\varepsilon +\nu) \| \nabla_{\V,\Ho}^2 f \|^2 \ge  \| \nabla_{\Ho}^2 f \|^2+  \left( \rho_1 -\frac{\kappa}{2\varepsilon+\nu} \right)  \| \nabla_\mathcal{H} f \|^2+\rho_2  \| \nabla_\mathcal{V} f \|^2 .
\]
From Cauchy-Schwarz inequality we get
\[
  \| \nabla_{\Ho}^2 f \|^2 +\varepsilon (\varepsilon +\nu)  \| \nabla_{\V}^2 f \|^2 \ge \frac{1}{n+\frac{m\varepsilon}{\nu+\varepsilon}} ( \Delta_\varepsilon f )^2,
\]
and the conclusion follows.
\end{proof}

The usual Bakry-\'Emery curvature dimension inequality is obtained when $\nu=0$. The Baudoin-Garofalo curvature dimension inequality \cite{BG} for the horizontal Laplacian is obtained when $\varepsilon =0$.

\end{document}